\documentclass[a4paper,12pt,reqno]{amsart}
\usepackage{latexsym}
\usepackage{amssymb} 
\usepackage{mathrsfs}
\usepackage{amsmath}
\usepackage{latexsym}
\usepackage{delarray}
\usepackage{amssymb,amsmath,amsfonts,amsthm,mathrsfs}
\usepackage{blkarray}

\usepackage{hyperref}
\renewcommand\eqref[1]{(\ref{#1})} 
 \newtheorem{thm}{Theorem}[section]
 \newtheorem{cor}[thm]{Corollary}
 \newtheorem{lem}[thm]{Lemma}
 \newtheorem{prop}[thm]{Proposition}
 \theoremstyle{definition}
 \newtheorem{defn}[thm]{Definition}
 \theoremstyle{remark}

 \numberwithin{equation}{section}

\newcommand{\half}{\frac{1}{2}}

\newcommand{\ene}{\mathbb{N}}
\newcommand{\bba}{\mathcal{B}}

\newcommand{\er}{\mathbb{R}}
\newcommand{\ce}{\mathbb{C}}
\newcommand{\zet}{\mathbb{Z}}
\newcommand{\Co}{\mathbb{C}}
\newcommand{\zn}{\mathbb{Z}^n}
\newcommand{\tn}{\mathbb{T}^n}

\newcommand{\To}{\mathbb{T}}

\newcommand{\fou}{\mathcal{F}}

\newcommand{\bi}{\begin{itemize}}

\newcommand{\ei}{\end{itemize}}
\newcommand{\be}{\begin{enumerate}}
\newcommand{\ee}{\end{enumerate}}
\newcommand{\beq}{\begin{equation}}
\newcommand{\ar}{\mathbb{R}}
\newcommand{\eq}{\end{equation}}

\newcommand{\jpxi}{\langle k \rangle}
\newcommand{\Dcal}{\mathcal{D}}
\newcommand{\Dp}{\mathcal{D}_p}

\def\Op{{{\rm Op}}}

\def\Hcal{{\mathcal H}}

\DeclareMathOperator{\Tr}{Tr}
\DeclareMathOperator{\Det}{Det}

\def\Rn{{{\mathbb R}^n}}
\def\Tn{{{\mathbb T}^n}}
\def\Zn{{{\mathbb Z}^n}}

\def\SU2{{{\rm SU(2)}}}
\def\lapsu2{{{\mathcal L}_\SU2}}

\def\Op{\text{\rm Op}}

\begin{document} 

%
%
%
%
%
%
%
%
%
\title[A Poincar\'e determinant on the torus]
 {A Poincar\'e determinant on the torus}

\author[Julio Delgado]{Julio Delgado}

\address{%
	Universidad del Valle
	Departamento de Matem\'aticas\\
	Calle 13 100-00\\
	Cali-Colombia}
\email{delgado.julio@correounivalle.edu.co}



\date{\today}
\begin{abstract}
In this work we introduce a  Poincar\'e determinant type for   operators on the torus $\To^n$. As an application we   establish the existence of nontrivial solutions for elliptic  equations of the form $(-\Delta)^{\frac{\nu}{2}}u+Qu=0$ on  $\To^n$ by using the Hill's method. 
\end{abstract}

\maketitle

\section{Introduction}
The  Poincar\'e determinant arises as an attempt to formulate a rigorous  definition for the  determinant of a matrix of infinite order. Infinite systems of linear equations  
 seems to appear for the first time in Fourier's work on the theory of heat. 
It is well known that determinants of matrices of infinite order were  introduced by  G. W. Hill in his investigation on the motion of  the lunar perigee(c.f \cite{hill1:cl}) in order to study  the equation that bear his name. However, Hill's approach was lack of mathematical rigorousness
 and the relevance of the applications motivated H. Poincar\'e to introduced a suitable class of matrices and their corresponding determinants (c.f \cite{poi:hill}). Poincar\'e's  approach to study the Hill's differential equation is elegant and  therein the space $\ell^1(\zet\times\zet)$ arises as the natural one for the definition of such determinants.  The main problem is to find conditions to ensure the existence of periodic solutions for the Hill's differential equation and the Poincar\'e determinant is useful for this purpose.\\
 

  A general point of view to define the determinant of $I+A$ consist in considering $A$ as an operator in a class endowed with a trace, that is is the point of view of Fredholm determinants. 
 There are several approaches to define traces and determinants in the setting of Banach spaces are known. Herein we  consider the point of view of embedded algebras introduced by I. Gohberg, S. Goldberg and N. Krupnik (cf. \cite{goh:trace}). We are going to work within the framework of embedded algebras and specifically the algebra of matrices in $\ell^1(\zet\times\zet)$ introduced by Poincar\'e for the study of the Hill's differential equation. \\


In this work we introduce a Poincar\'e determinant type on the torus $\To^n$ by using the Fourier transform on this group. This will allow us to establish the existence of nontrivial solutions for elliptic equations of the form $(-\Delta)^{\frac{\nu}{2}}u+Qu=0$ on  $\To^n$ for $\nu>n$, by applying  an extension of the Hill's method, in the special case of pseudodifferental operators on $\Tn$. \\

The study of Fredholm determinants has been an active field of research, in particular
 due to its applications in the analysis of differential equations,
 see e.g. \cite{Zhao-Barnett,Bothner-Its:CMP-2014, Borodin-Corwin-Remenik,Gesztesy-Latushkin-Zumbrun,McKean:CPAM-2003}. The relevance of such applications has also attracted the attention towards
 the numerical analysis of such determinants.  A systematic study of numerical computations for Fredholm determinants  was initiated by Bornemann \cite{Bornemann:MC-2010}.  Formulas for the trace and Schatten-von Neumann properties  in different settings have been studied in   \cite{dr:suffkernel}, \cite{dr14a:fsymbsch}, \cite{dr13:schatten} and \cite{dr13a:nuclp}.\\\\

 In Section 2 we briefly recall some basics on traces and determinants in the setting introduced in \cite{goh:trace}. In Section 3 we define a Poincar\'e determinant for operators on the torus. We establish the existence of nontrivial solutions for the equations $(-\Delta)^{\frac{\nu}{2}}u+Qu=0$ on the torus $\To^n$ by using the Hill's method with the Poincar\'e determinant.
 
 
\section{traces and Determinants}
In this section we recall the definition of traces and determinants in the sense of \cite{goh:trace} and some of their basic properties. The main idea consists in construct an extension from finite rank operators to some suitable algebra.\\

Let  $\mathcal{B}$ be a Banach space, we denote by  $\mathfrak{F}(\bba)$ the space of finite rank operators on  $\bba$. We recall that in this context finite rank operators are assumed bounded. 
  We also denote by $\mathcal{L}(\bba)$ the $C^*$-algebra of bounded linear operators on $\bba$.  
 We now briefly recall the definition of trace and determinants for finite rank operators. The following fundamental properties of trace and determinants for finite rank operators are mainly consequences of the finite square matrix setting. We refer the reader to 
  \cite{goh:trace} for a more comprehensive treatment on those topics. 
  
\begin{lem} Let $\bba$ be a Banach space and $F\in \mathfrak{F}(\bba)$. Then  $\bba$ can be decomposed as a direct sum
	\beq B=M_F\oplus N_F ,\label{dmnfg}\eq
	where $F(M_F)\subset M_F$ and $N_F\subset Ker F.$
\end{lem}

The decomposition \eqref{dmnfg} allows us to write the operators $F$ and $I+F$ as $2\times 2$ matrices of the form
\beq F=\begin{bmatrix} 
	F_1 & 0 \\
	0 & 0 
\end{bmatrix},\,\, I+F=\begin{bmatrix} 
I_1+F_1 & 0 \\
0 & I_2 
\end{bmatrix}.
\label{mm35}\eq
Here $F_1=F|_{M_F}$ is the restriction of $F$ to the finite dimensional subspace $M_F$ and $I, I_1, I_2$ denote the identity operators in $\bba, M_F, N_F$ respectively. Since  $M_F$ is a finite dimensional space, the functionals $\Tr{F_1}$ and $\Det(I_1+F_1)$ are well defined. Using this we can define:
\beq \Tr(F):=\Tr(F_1), \,\, \Det (I+F):=\Det(I_1+F_1).\label{ddtrf5}
\eq
These definitions are independent of the choice of the subspace $M_F$. Indeed, this is due to the following important formulas in terms of eigenvalues:
\beq \Tr(F_1)= \sum\limits_{j=1}^m \lambda_j(F_1),\,\, \Det(I+F_1)=\prod_{j=1}^{m}(1+\lambda_j(F_1)) ,  
\eq
where $m=\dim M_F$ and $\lambda_1(F_1),\dots,\lambda_m(F_1)$ are the eigenvalues of $F_1$ counted according to their algebraic multiplicites. Moreover, an application of \eqref{mm35} and the Jordan decomposition of $F_1$ shows that the nonzero eigenvalues of $F$ do not depend on the choice of $M_F$ and  
\beq \Tr(F)= \sum\limits_{j=1}^m \lambda_j(F),\,\, \Det(I+F)=\prod_{j=1}^{m}(1+\lambda_j(F)).  
\eq
Therefore the Definition \eqref{ddtrf5} is independent of the choice of the subspace $M_F$. Moreover, we see that $I+F$ is invertible in $\mathcal{L}(\bba)$ if and only if $\Det(I+F)\neq 0$.\\ 

The functionals $\Tr(\cdot)$ and $\Det(\cdot)$ enjoy some fundamental properties:\\

The trace  $\Tr(\cdot)$ is linear on $\mathfrak{F}(\bba)$ and 
\beq\Tr(AB)=\Tr(BA).\label{comtr}\eq
The determinant  has the following multiplicative property
\beq \Det((I+A)(I+B))=\Det(I+A)\Det(I+B)
\eq
and 
\beq \Det(I+AB)=\Det(I+BA).
\eq


In order to extend the trace and determinants to larger classes of operators, we recall the notion of algebras that we are going to use. We note that the functionals trace and determinant are not continuous with respect to the operator norm. This fact motivates 
 the introduction of the following concept of algebra. 

\begin{defn} A subalgebra $\mathcal{D}$ of $\mathcal{L}(\bba)$ is {\em  continuously embedded } in $\mathcal{L}(\bba)$ if the following two conditions hold:\\
	
	(i) There exists a norm $\|\cdot\|_{\mathcal{D}}$ on $\mathcal{D}$ and a constant $C>0$ such that
	\beq\|A\|_{\mathcal{L}(\bba)}\leq C\|A\|_{\mathcal{D}}\eq
	for all $A\in \mathcal{D}$.\\
	
	(ii) $\|AB\|_{\mathcal{D}}\leq \|A\|_{\mathcal{D}}\|B\|_{\mathcal{D}}$ for all  $A, B\in \mathcal{D}$.
	\end{defn}

For the sake of simplicity, we say that the subalgebra $\mathcal{D}$ is an {\em embedded subalgebra} if the norm on  $\mathcal{D}$ satisfies (i) and (ii).\\

If, in addition, $\mathfrak{F}_{\mathcal{D}}:=\mathfrak{F}(\bba)\bigcap\mathcal{D}$ is dense in $\mathcal{D}$ with respect to the norm $\|\cdot\|_{\mathcal{D}}$, we say that $\mathcal{D}$ is {\em approximable}. \\

If $\mathcal{D}$ is an approximable algebra we can continuously extend the trace and determinant from finite rank operators to the algebra $\mathcal{D}$. 

\begin{thm}\label{th2eqr} Let $\mathcal{D}\subset\mathcal{L}(\bba)$ be an approximable embedded subalgebra. The following statements are equivalent:

(i) The function $\Det(I+F):\mathfrak{F}_{\mathcal{D}}\rightarrow\ce$ admits a continuous extension in the $\mathcal{D}$-norm from $\mathfrak{F}_{\mathcal{D}}$ to $\mathcal{D}$.\\

 (ii) The linear functional $\Tr(F) $ is bounded in the $\mathcal{D}$-norm on $\mathfrak{F}_{\mathcal{D}}$.
\end{thm}

The conditions above are satisfied in a good number of important examples. First we point out  how the conditions above are used to extend the trace and determinants from finite rank operators.\\

Let $\mathcal{D}\subset\mathcal{L}(\bba)$ be an approximable embedded subalgebra and assume  $\mathcal{D}$
 satisfies one of the equivalent conditions of Theorem \ref{th2eqr}. If $A\in \mathcal{D}$ we define the {\em trace} of $A$ and the {\em determinant } $I+A$ in the algebra $ \mathcal{D}$ by the equalities:
 \beq \Tr_{\mathcal{D}}(A):=\lim\limits_{n\rightarrow\infty}\Tr(F_n) \mbox{ and } \Det\limits_{\mathcal{D}}(I+A):=\lim\limits_{n\rightarrow\infty}\Det(I+F_n) ,  \eq
 where $\|A-F_n\|_{\mathcal{D}}\rightarrow 0$ with $F_n\in \mathfrak{F}_{\mathcal{D}}(\bba)$. As a consequence of the continuity of the extended trace and determinant, it can be shown that these definitions are independent of the choice of the approximative sequence $F_n$.\\

The following theorem will be useful to establish the existence of solutions for homogeneous differential equations. 

\begin{thm}\label{th2eqra} Let $\mathcal{D}\subset\mathcal{L}(\bba)$ be an approximable embedded subalgebra. Suppose that 
the function $\Det(I+F):\mathfrak{F}_{\mathcal{D}}\rightarrow\ce$ admits a continuous extension in the $\mathcal{D}$-norm from $\mathfrak{F}_{\mathcal{D}}$ to $\mathcal{D}$. Then, an operator $I+A$
with $A\in \mathcal{D}$ is invertible in $\mathcal{L}(\bba)$ if and only if $\Det_{\mathcal{D}}(I+A)\neq 0.$
	\end{thm}



\section{The Poincar\'e Determinant for operators on the torus}
In this section we introduce the Poincar\'e determinant for a suitable algebra of operators on the torus $\To^n$ and the the main results of this work. 
We first define a suitable algebra of operators on $\ell^{p}(\Zn)$ spaces. \\

\begin{defn} Let $\ell^{1}(\Zn\times\Zn)$ be the space of matrices $A=(a_{jk})_{(j,k)\in  \Zn\times\Zn}$ with complex entries  such that 
\beq\sum\limits_{(j,k)\in  \Zn\times\Zn}|a_{jk}|<\infty.\label{inpoi}\eq 
We endow $\ell^{1}(\Zn\times\Zn)$  with the  $\ell^{1}$-norm: 
\[ \|A\|_{\ell^{1}}:=\sum\limits_{(j,k)\in\in  \Zn\times\Zn}|a_{jk}|.\]

\end{defn}
We now note that if $A\in \ell^{1}(\Zn\times\Zn) $, then we also have $^tA\in \ell^{1}(\Zn\times\Zn)$, where  $^tA$ denote the transpose of $A$ defined by $^tA=(a_{kj})_{(j,k)\in  \Zn\times\Zn}$. Moreover we have  $\|A\|_{\ell^{1}}=\|^tA\|_{\ell^{1}}.$\\

\begin{thm} Let $1\leq p<\infty$ and $A\in\ell^{1}(\Zn\times\Zn)$. If 
$x\in \ell^{p}(\Zn)$ we define $\tilde{A}x:\Zn\rightarrow\ce$ by 
\[(\tilde{A}x)_k:= \sum\limits_{j\in\Zn}a_{jk}x_j.\]
Then $(\tilde{A}x)_k$ is well-defined, $\tilde{A}x\in  \ell^{p}(\Zn)$, $x\rightarrow \tilde{A}x$ defines a bounded operator from $\ell^{p}(\Zn)$ into $\ell^{p}(\Zn)$ and 
\[\|\tilde{A}\|\leq \|A\|_{\ell^{1}}.\]
\end{thm}
\begin{proof} Let $p'$ be such that $\frac{1}{p}+\frac{1}{p'}=1$. We observe that the matrix $^tA$ is the kernel of the integral operator $\tilde{A}$ with respect to the counting measure on $\mathcal{P}(\Zn)$. On the other hand, since 
\[ \sum\limits_{j\in  \Zn}|a_{jk}|,\,\,\sum\limits_{k\in  \Zn}|a_{jk}|\leq \sum\limits_{(j,k)\in  \Zn\times\Zn}|a_{jk}|=\|A\|_{\ell^{1}}.\]
Then the conclusions follows from the Schur test for integral operators and we also have
$\|\tilde{A}\|\leq\|A\|_{\ell^{1}}^{\frac{1}{p}+\frac{1}{p'}}=\|A\|_{\ell^{1}}.$
\end{proof}
We  note that if  $A\in\ell^{1}(\Zn\times\Zn)$, then  $A$ is also the matrix of the operator $\tilde{A}$ with respect to the standard basis.  

\begin{defn} Let $1\leq p<\infty$.  We denote by $\Dcal_p$ the algebra of all operators $\tilde{A}$ on $\ell^p(\Zn)$ where its matrix $A$ with respect to the standard basis   belongs to $\ell^{1}(\Zn\times\Zn)$. We endow $\Dcal_p$ with the norm
\[\|\tilde{A}\|_{\Dp}=\sum\limits_{(j,k)\in\Zn\times\Zn}|a_{jk}|.\]
The value $p$ indicates that the operator is considered acting on the space $\bba=\ell^{p}(\Zn)$ despite the norm is independent of $p$. Henceforth we identify the operator $\tilde{A}$ in $\Dcal_p$ with the matrix $A$. That $\Dcal_p$ is indeed an algebra follows from the fact that  $\ell^{1}(\Zn\times\Zn)$ is an algebra, and it will be clarified below in the next proposition.
\end{defn}

\begin{prop} Let $p$ be such that $1\leq p<\infty$. The space  $\Dp$ is  an approximable embedded subalgebra of $\mathcal{L}(\bba)$ for  $\bba=\ell^{p}(\Zn)$. Moreover, the  determinant $\Det(I+F)$ can be continuously extended to $\Dp$
\end{prop}
\begin{proof} To see that $\Dp$ is a subalgebra of $\mathcal{L}(\bba)$. Let $A,B\in \ell^{1}(\Zn\times\Zn)$. We note that, since $\ell^{1}(\Zn\times\Zn)\subset\ell^{2}(\Zn\times\Zn)$ and  by Cauchy-Schwarz inequality we have
\[\sum\limits_{(k,l)\in\Zn\times\Zn}|\sum\limits_{i\in\Zn}a_{ki}b_{il}|\leq \sum\limits_{(k,l)\in\Zn\times\Zn}\sum\limits_{i\in\Zn}|a_{ki}b_{il}|\leq \|A\|_{\ell^{2}}\|B\|_{\ell^{2}} .\]
Thus, the multiplication $AB$ is well defined,  $\ell^{1}(\Zn\times\Zn)$ is an algebra and consequently $\Dp$ is a subalgebra of $\mathcal{L}(\bba)$.\\

Now, let $\tilde{A}\in \Dp $, $\epsilon >0$, and $A=(a_{jk})_{j,k\in  \Zn\times\Zn}$ the matrix  corresponding to $\tilde{A}$. There exists a finite set $G\subset \Zn\times\Zn $ such that $\sum\limits_{(j,k)\in H^\complement}|a_{jk}|<\epsilon,$  provided $H\supseteq G$. We define the matrix  $A_G:=(a_{jk})_{(j,k)\in G^\complement}$, where $G^\complement$ is the complement set of $G$ . Then, the corresponding operator $F:=\widetilde{A_G}$ is a bounded finite rank operator and 

\[\|\tilde{A}-F\|_{\Dp}<\epsilon .\]

On the other hand, we observe that
\beq |\Tr(A)|=|\sum\limits_{j\in\Zn}a_{jj}|\leq \|A\|_{\Dp}.\eq
Thus, $\Dp$ satisfies the assumption (ii) of Theorem \ref{th2eqr}, and therefore the determinant $\Det(I+F)$ can be continuously extended to $\Dp$. 
\end{proof}
We call the above determinant on $\Dp$ the {\em Poincar\'e determinant on} $\Dp$, or just the {\em Poincar\'e determinant } when there is not room for ambiguity. \\

%

We now consider operators on the torus $\Tn$. In this special domain we can  take advantage of the Fourier transform and the duality with $\Zn$ to relate the operators on the torus with the operators on  $\Dp$. By doing so we can define a  Poincar\'e determinant type for operators on the torus. From now on, we will focus on the special case $p=2$ to take advantage of the isometry of the Fourier transform.
The collection $\{e^{2\pi ix\cdot k}\}_{k\in\Zn}$ is an orthonormal basis of $L^2(\Tn)$. We denote the Fourier transform on the torus $\Tn$ by $\fou_{\Tn}$ which is defined by
\[(\fou_{\Tn}\varphi)(k)=\int\limits_{\Tn}e^{-2\pi ix\cdot k}\varphi(x)dx,\]
for all $\varphi\in C^{\infty}(\Tn).$
\begin{thm} Let $\Gamma:\Dcal_2\longrightarrow\mathcal{L}(L^2(\Tn))$ be the mapping defined by \[\Gamma(A):=\fou_{\Tn}^{-1}A\fou_{\Tn},\]
for every $A\in \Dcal_2.$ Then, the following properties hold:\\
\begin{itemize}
\item[(i)] $\Gamma$ is  an  algebra isomorphism between $\Dcal_2$ and  $\Gamma(\Dcal_2)$. By considering the norm on $\Gamma(\Dcal_2)$ induced from $\Dcal_2$, we have an isometric isomorphism. 
\item[(ii)]  $\Gamma(\Dcal_2)$ is an approximable embedded algebra and  for every operator $T\in \Gamma(\Dcal_2)$ by using the algebra isomorphism $\Gamma$ we can define a trace for $T$ and a determinant for $I+T$.

\end{itemize}
 
\end{thm}
\begin{proof} (i) It is clear that $\Gamma $ is linear. To see  that $\Gamma$ preserves the multiplication we observe that, for $A, B\in \Dcal_2$
 we have 
\begin{align*}\Gamma(AB)&= \fou_{\Tn}^{-1}AB\fou_{\Tn}\\
&= \fou_{\Tn}^{-1}A\fou_{\Tn}\fou_{\Tn}^{-1}B\fou_{\Tn}\\
&=( \fou_{\Tn}^{-1}A\fou_{\Tn})(\fou_{\Tn}^{-1}B\fou_{\Tn})\\
&=\Gamma(A)\Gamma(B). 
\end{align*}
The remaining part of (i) follows immediately.\\

(ii) In order to prove that $\Gamma(\Dcal_2)$ is an approximable embedded algebra, we point out that the isomorphism $\Gamma$ preserves the class of finite rank operators $\mathfrak{F}_{\Dcal_2}(\ell^2)=\mathfrak{F}(\ell^2)$, that is   \[\Gamma(\mathfrak{F}(\ell^2))=\mathfrak{F}(L^2(\Tn)).\]
The above holds since given $G\in\mathfrak{F}(L^2(\Tn))$, we can write
\[G=\fou_{\Tn}^{-1}(\fou_{\Tn}G\fou_{\Tn}^{-1})\fou_{\Tn}.\]
Hence $\Gamma(\fou_{\Tn}G\fou_{\Tn}^{-1})=G$.\\

Since the Poincar\'e determinant is defined on  $\Dcal_2$ as well as the trace, then via the isomorphism $\Gamma$ the same can be done on the embedded subalgebra $\Gamma(\Dcal_2)$ of operators on $L^2(\Tn)$. \\

Therefore $\Gamma(\Dcal_2)$ is an approximable embedded algebra and  for every operator $T\in \Gamma(\Dcal_2)$ by using the algebra isomorphism $\Gamma$ we can define a trace for $T$ and a determinant for $I+T$. \end{proof}
The theorem above justifies the following definition.
\begin{defn} We will call  the induced determinant on  $\Gamma(\Dcal_2)$, the {\em Poincar\'e determinant }, for operators on  $L^2(\Tn)$, but to distinguish the domains we will denote it by $\Det_{\Gamma}$. Thus, for $T\in \Gamma(\Dcal_2)$ we define 
$\Det_{\Gamma}(I+T)$ by 
\beq\label{poidef57} \Det_{\Gamma}(I+T):=\Det(I+A),\eq
where $T=\fou_{\Tn}^{-1}A\fou_{\Tn}$.\\
\end{defn}

In the special case of pseudodifferential operators on the torus $\Tn$   we  establish explicit formulas in terms of the symbol from the isomorphism $\Gamma$. We first recall some basic definitions and properties.\\

\begin{defn} {\bf (Finite differences $\Delta_{k}^{\alpha}$).} Let $\sigma:\zn\rightarrow \Co$ and $1\leq i,j\leq n.$ Let $\delta_j\in \ene_0^n$ be defined  by
\begin{equation}
 (\delta_j)_i:=\left\{
\begin{array}{rl}
1,& \mbox{ if } i=j\\
0,& \mbox{ if } i\neq j.
\end{array} \right.
\label{krona}\end{equation}

We define the {\em forward partial difference operator}   $\Delta_{\xi_j}$ by
\[\Delta_{k_j}\sigma(k)=\sigma(k+\delta_j)-\sigma(k),\] 
and for $\alpha\in \ene_0^n$ define
\[\Delta_{k}^{\alpha}:=\Delta_{k_1}^{\alpha_1}\cdots\Delta_{k_n}^{\alpha_n}.\]
\end{defn}

\begin{prop}\label{combal} {\bf (Formulae for  $\Delta_{k}^{\alpha}$)} Let $\phi:\zn\rightarrow \Co$. We have 
\beq\Delta_{k}^{\alpha}\phi(k)\,=\, \sum\limits_{\beta\leq\alpha}(-1)^{|\alpha-\beta|}{\alpha\choose\beta} \phi(k+\beta).\label{cohr}\eq
\end{prop}

We now recall the definition of the  toroidal symbol classes. 
We  define $\jpxi :=(1+|k|^2)^\half$ for $k\in\zn$, with $|k|^2=k_1^2+\cdots +k_n^2$.\\

\begin{defn} Let $m\in\ar$. Then the {\em toroidal symbol class} $S_{1,0}^m(\mathbb{T}^n\times \zet^n)$ consists of those functions $a(x, k)$ which are smooth in $x$ for all $k\in\zn$, and which satisfy {\em toroidal symbol inequalities} 
\beq |\Delta_{k}^{\alpha}\partial_x^{\beta}a(x,k)|\leq C_{\alpha\beta}\jpxi^{m-|\alpha|}\eq
for every $x\in\tn, k\in\zn$ and for all $\alpha, \beta\in \ene_0^n$.
\end{defn} 

The toroidal quantisation associated to a symbol $\sigma\in S_{1,0}^m(\mathbb{T}^n\times \zet^n)$ is the densely defined  operator $T=\sigma(x,D)$ given by 
\[ \,\,\,\,  \,\,\,\,\,\,\,\,\,\,\,\,\,\,\,\,Tf(x)=\sum\limits_{k\in\zn}e^{2\pi ix\cdot k}\sigma(x,k)\widehat{f}(k) ,\,\,\, f\in C^\infty(\tn).\]
On the other hand if one has a continuous linear operator $T:C^\infty(\tn)\rightarrow C^\infty(\tn) $, its symbol $\sigma(x,k)$ can be recovered from the following formula
\beq \sigma(x,\xi)=e^{-i2\pi x\cdot k}Te_{k}(x),\eq
where for every $k\in\zn$, $e_{k}(x)=e^{i2\pi x\cdot k},$ for all $x\in\tn$.\\

The family of pseudodifferential operators corresponding to the class of symbols $\sigma\in S_{1,0}^m(\mathbb{T}^n\times \zet^n)$ will be denoted by $\Op S_{1,0}^{m}(\mathbb{T}^n\times \zet^n)$. The toroidal  quantisation has been extensively analysed in \cite{rt:book} for the general case of $\Tn$ and on compact Lie groups. For the toroidal H\"ormander class of order $m\in \mathbb{R},$ one has  $\Psi^m(\mathbb{T}^n,\textnormal{loc})=\{\sigma(x,D):\sigma\in S_{1,0}^m(\mathbb{T}^n\times \zet^n)\}$ (cf. \cite{rt:book}). \\

 We observe that for $f\in C^{\infty}(\Tn)$, and after a change of variable we have 
 
\begin{align*}
Tf(x)=&\sum\limits_{k\in\Zn}e^{2\pi ix\cdot k}\sigma(x,k)\widehat{f}(k)\\
=&\sum\limits_{k\in\Zn}\left(\sum\limits_{l\in\Zn}\widehat{\sigma}(l,k)e^{2\pi ix\cdot l}\right)\widehat{f}(k)e^{2\pi ix\cdot k}\\
=&\sum\limits_{j\in\Zn}\left(\sum\limits_{k\in\Zn}\widehat{\sigma}(j-k,k)\widehat{f}(k)\right) e^{2\pi ix\cdot j}.
\end{align*}
The Fourier transform on $\sigma$ should be understood with respect to the space-like variable on the torus $\Tn$. From the last identity we can write $Tf$ as the Fourier series
\[Tf(x)=\sum\limits_{j\in\Zn}C_{j} e^{2\pi ix\cdot j},\]
where $C_j$ is the $j$-th Fourier coefficient of $Tf$.\\

We  now note that since  
\[C_j=\sum\limits_{k\in\Zn}\widehat{\sigma}(j-k,k)\widehat{f}(k).\]
Then we can write 
\[C_j=\sum\limits_{k\in\Zn}A_{jk}\widehat{f}(k)=\sum\limits_{k\in\Zn}A_{jk}\phi(k),\]
where $A_{jk}$ is the matrix defined by 
\beq\label{amatf8} A_{jk}=\widehat{\sigma}(j-k,k)\eq
and $\phi(k)=\widehat{f}(k).$\\

We should now ensure the convergence of the above series. We note that 
if $\sigma(\cdot, k)\in L^1(\Tn\times\Zn)$ for every $k$ and since $L^2(\Tn)\subset L^1(\Tn)$, the above calculations can be justified. These conditions are granted if $\sigma$ is continuous on the torus, which is the case of $ S^m(\Tn\times\Zn)$.  \\

Therefore if $\sigma(\cdot, k)\in L^1(\Tn\times\Zn)$ for every $k$, the identity 
\eqref{amatf8} establishes an explicit formula for $A\in \Dcal_2$ in terms of the symbol provided 
$\widehat{\sigma}\in\ell^1(\Zn\times\Zn)$, where we recall the Fourier transform is taken with respect to the space-like variable on the torus.


\begin{cor}\label{cor49h} Let $\sigma\in  S_{1,0}^{m}(\Tn\times\Zn)$ with $m<-n$. Then $T\in \Dcal_2(\Tn)$.
\end{cor}
\begin{proof} Since $\sigma\in  S_{1,0}^{m}(\Tn\times\Zn)$ with $m<-n$. 
By Lemma 4.2.1 of \cite{rt:book}, for all $r\in \mathbb{R},$ there exists $C_r>0,$ such that
\begin{equation*}
    |\widehat{\sigma}(j,k)|\leq C_{r}(1+|j|)^{-r}(1+|k|)^{m}.
\end{equation*}If $m<-n,$ and $r>n,$ then
\begin{align*}
    \Vert T\Vert_{\Gamma(\mathcal{D}_2)}&=\sum_{j,k\in \Zn}|A_{jk}|=\sum_{j,k\in \Zn}|\widehat{\sigma}(j-k,k)|\\
     &\lesssim \sum_{j,k\in \Zn}(1+|j-k|)^{-r}(1+|k|)^{m}<\infty.
\end{align*}
Therefore the identity \eqref{amatf8} holds and $T\in \Dcal_2(\Tn)$.\\
 
\end{proof}


We now consider operators of the form $(-\Delta)^{\frac{\nu}{2}}+Q(x)$ on the torus $\To^n$ within the setting of toroidal quantization and global symbols above described.
\begin{defn}
Let  $\nu$ be a strictly positive real number, the fractional Laplacian $(-\Delta)^{\frac{\nu}{2}}$ on the torus $\tn$ is defined as the  Fourier multiplier corresponding to  $(2\pi)^{\nu}|k|^{\nu}$, that is 
\beq \widehat{(-\Delta)^{\frac{\nu}{2}}u}(k)=(2\pi)^{\nu}|k|^{\nu}\widehat{u}(k),\eq 
 for every $k\in\Zn$ and $u\in C^{\infty}(\Tn)$.
\end{defn} 
  Other alternative definitions are possible as in the case of $\Rn$.  We point out that on $\Rn$, the fractional Laplacian $(-\Delta)^{\frac{\nu}{2}}$ has not a symbol in a class of pseudodifferential operators unless $\frac{\nu}{2}$ be an integer, however in our toroidal setting  the fractional Laplacian will be a pseudodifferential operator for every $\nu>0.$ \\
 
 In order to clarify such property for the fractional Laplacian on the torus we state the following lemma.
 
\begin{lem}\label{flle} Let $\nu$ be a real number $>0$. Then $(-\Delta)^{\frac{\nu}{2}}\in \Op S_{1,0}^{\nu}(\mathbb{T}^n\times \zet^n)$. 
\end{lem}
 \begin{proof} We should show that the function $\sigma:\zn\rightarrow\er$ defined by $\sigma(k)=|k|^{\nu}$ belongs to  $S_{1,0}^{\nu}(\mathbb{T}^n\times \Zn)$. We choose a function $\chi\in C^{\infty}(\Rn)\,$ such that
 \begin{equation}
 \chi(k):=\left\{
\begin{array}{rl}
1,& \mbox{ if } |k|\geq 1,\\
0,& \mbox{ if } |k|\leq \half.
\end{array} \right.
\label{ht56e}\end{equation}
We observe that  $\widetilde{\sigma}(k):=\chi(k)|k|^{\nu}$ is smooth on $\Rn$ and  $\widetilde{\sigma}\in S_{1,0}^{\nu}(\mathbb{T}^n\times \Rn)$. Hence $\widetilde{\sigma}(x,D)\in \Op S_{1,0}^{\nu}(\mathbb{T}^n\times \Rn)$ and by  Corollary 4.6.13 of  \cite{rt:book}, we obtain that $\widetilde{\sigma}(x,D)\in \Op S_{1,0}^{\nu}(\mathbb{T}^n\times \zn)$. 
On the other hand we have that $\chi(k)|k|^{\nu}=|k|^{\nu}=\sigma(k)$ for all $k\in\zn$, and  therefore $\sigma\in S_{1,0}^{\nu}(\mathbb{T}^n\times \zet^n)$.\\
\end{proof}
 
We first recall the definition of ellipticity in this setting. 
\begin{defn} Let $m\in\er$ and let $\sigma$ be a symbol in $\sigma\in S^m(\To^n\times\zet^n)$. We say that the corresponding periodic pseudodifferential operator $T_{\sigma}$ is {\em elliptic of order }$m$, if $\sigma$ satisfies
\beq \forall (x,k)\in \To^n\times\zet^n:\,\, |k|\geq n_0\implies |\sigma(x,k)|\geq C_0\langle k\rangle^m\eq  
for some constants $n_0, c_0>0$.\\

We will also say that $\sigma(x,D)$ is a {\em strongly elliptic operator} 
if $\sigma$ satisfies:
\beq \forall (x,k)\in \To^n\times\zet^n:\,\, |k|\geq n_0\implies Re\,\sigma(x,k)\geq C_0\langle k\rangle^m\eq  
for some constants $n_0, c_0>0$.\\
\end{defn}

Since the fractional Laplacian  $(-\Delta)^{\frac{\nu}{2}}$ is an element of  $\Op S_{1,0}^{\nu}(\mathbb{T}^n\times \zn)$ for any $\nu>0$. Then it is clear that it is also an  strongly elliptic operator of order $\nu$. \\

We now recall the definition of the  Sobolev space of order $s\in\er$ on the torus.
 For $u\in\mathcal{D}'(\To^n)$ we define the norm $\|\cdot\|_{H^s(\To^n)}$ by
\[\|u\|_{H^s(\To^n)}:=\left(\sum\limits_{k\in\zet^n}(1+|k|^2)^s|\hat{u}(k)|^2\right).\] 
 The Sobolev space $H^s(\To^n)$ is the space of the $1$-periodic distributions $u$ such that  $\|u\|_{H^s(\To^n)}<\infty.$\\
  
In the case of an operator $(-\Delta)^{\frac{\nu}{2}}+Q(x)$ we have the following result regarding  ellipticity and Fredholmness. We note that for $\nu=2$ we have a Schr\"odinger operator on $\To^n$. 
\begin{thm} Let $\nu$ be a real number $>0$.  If $Q$ is a $C^{\infty}$ function on $\To^n$. Then 
$(-\Delta)^{\frac{\nu}{2}} +Q(x)$ is an strongly elliptic operator, and consequently a Fredholm operator from $H^s(\To^n)$ into $H^{s-\nu}(\To^n)$ for all $s\in\er$.\\
\end{thm}
\begin{proof}  We write $T=(-\Delta)^{\frac{\nu}{2}} +Q(x)$, then $\sigma_T(x,k)=2^{\nu}\pi^{\nu}|k|^{\nu}+Q(x)$.
 Since $Q$  is continuous on the torus then $Re\, Q$  it is bounded. We set $C_1=\min_{x\in\Tn} Re\, Q(x)$ and  observe that
\begin{align*}
Re\,\sigma_T(x,k)= &2^{\nu}\pi^{\nu}|k|^{\nu}+Re\,Q(x)\\
\geq &2^{\nu}\pi^{\nu}|k|^{\nu}+C_1\\
\geq & C(|k|+1)^{\nu},
\end{align*}
for $|k|$ large enough. Thus, $T$ is strongly elliptic and therefore elliptic.  An application of Theorem 4.9.17 of \cite{rt:book},  ensuring that ellipticity of a pseudodifferential operator on the torus implies Fredholmness, concludes the proof.
\end{proof}
We now consider the elliptic operator $(-\Delta)^{\frac{\nu}{2}}+Q(x)$ on the torus $\To^n$ and study the existence of solutions for the following equation : 
\beq (-\Delta)^{\frac{\nu}{2}}u+Qu=0\label{shil87}\eq
 where 
\beq Q(x)=\sum\limits_{k\in\zet^n}e^{2\pi ix\cdot k}g_k\label{qc4}\eq
with
\beq \sum\limits_{k\in\zet^n}|g_k|<\infty.\label{qc4h}\eq
 In terms of the Poincar\'e determinant we have the following theorem. We recall that for the $1$-dimensional case this equation recovers the Hill's differential equation in the setting of periodic solutions.

\begin{thm} Let $\nu$ be real number $>n$. The equation \eqref{shil87}
has a solution of the form
\beq u(x)=\sum\limits_{k\in\zet^n}e^{2\pi ix\cdot k}b_k\label{sols7}\eq
with 
\beq\label{condb1} \sum\limits_{k\in\zet^n}|k|^{\nu}|b_k|<\infty\eq
if and only if 
\beq\Det(I+\tilde{A})=0,\eq
where $\tilde{A}$ is the operator from $\ell^{2}(\zet^n)$ into $\ell^{2}(\zet^n)$ corresponding to 
\beq A_{k,m}=\frac{g_{k-m}}{(2\pi)^{\nu}|k|^{\nu}+1}\label{8thj}. \eq

\end{thm}
\begin{proof}
We plug \eqref{qc4} and \eqref{sols7} into   \eqref{shil87}
and we have:\\
\begin{align*}
\sum\limits_{k\in\zet^n}(2\pi)^{\nu}|k|^{\nu}b_ke^{2\pi ix\cdot k}&\\+\left(\sum\limits_{m\in\zet^n}g_me^{2\pi ix\cdot m}\right)\left(\sum\limits_{k\in\zet^n}b_ke^{2\pi ix\cdot k}\right)=0.
\end{align*}
Hence we obtain
\beq (2\pi)^{\nu}|k|^{\nu}b_k+\sum\limits_{m\in\zet^n}g_{k-m}b_m=0,\label{hom3a}\eq
 which after dividing by $(2\pi)^{\nu}|k|^{\nu}+1$, is equivalent to 
\beq \label{3eqa}(I+\tilde{A})b=0,\eq 
 where $A$ is defined by \eqref{8thj}. It is not difficult to see from the condition $\nu>n$ that $A\in  \ell^{1}(\zet^n\times\zet^n)$, hence the determinant $\Det(I+\tilde{A})$ is well defined. Therefore, by  Theorem \ref{th2eqra} the equation \eqref{3eqa} has a non-trivial solution in $\ell^{1}(\zet^n)$ if and only if $\Det(I+\tilde{A})=0$.\\
 
We now complete the proof of the theorem assuming that $b\in\ell^1(\zet^n)$ is a solution of \eqref{3eqa} and showing that it satisfies the condition \eqref{condb1}. From \eqref{hom3a} and \eqref{qc4h} we obtain
\[\sum\limits_{k\in\zet^n}|k|^{\nu}|b_k|\leq\frac{1}{(2\pi)^{\nu}}\sum\limits_{k\in\zet^n}\sum\limits_{m\in\zet^n}|g_{k-m}b_m|\leq\frac{1}{(2\pi)^{\nu}}\left(\sum\limits_{m\in\zet^n}|b_m|\right)\left(\sum\limits_{k\in\zet^n}|g_{k}|\right)<\infty.\]

\end{proof}


\noindent{\bf{Data availability statement}}\\  All data generated or analysed during this study are included in this manuscript.\\

\noindent{\bf{Acknowledgements}}\\

\noindent The  author was supported by Grant CI-71234 Vic. Inv. Universidad del Valle. 
I would also like to thank an anonymous referee for the careful review. \\%



\end{document}